\documentclass[reqno]{amsart}
\usepackage{amssymb}
\usepackage{amsthm}
\usepackage{bbm}
\usepackage{graphicx}

\usepackage[colorlinks=true]{hyperref}
\usepackage[all]{xypic}

\usepackage[final]{epsfig}
\usepackage{psfrag}
\usepackage{epstopdf}

\newtheorem{thm}{Theorem}[section]
\newtheorem{theorem}[thm]{Theorem}

\theoremstyle{definition}

\theoremstyle{observation}

\theoremstyle{definition}

\newcommand{\A}{\mathcal{A}}

\newcommand{\M}{\mathcal{M}}

\newcommand{\cC}{\mathcal{C}}

\newcommand{\defin}[1]{{\it #1}}

\newcommand{\N}{\mathbb{N}}
\newcommand{\Q}{\mathbb{Q}}
\newcommand{\C}{\mathbb{C}}

{\bf}{\it}
{\bf}{\it}
\newtheorem*{riemann-mapping-theorem}{Riemann Mapping  Theorem}{\bf}{\it}
{\bf}{\it}
{\bf}{\it}
{\bf}{\it}
{\bf}{\it}
{\bf}{\it}
{\bf}{\it}
{\bf}{\it}
{\bf}{\it}
{\bf}{\it}

\newenvironment{pf*}[1]{\proof[#1]}{\endproof}
\usepackage{euscript}

\newcommand{\cal}[1]{{\mathcal #1}}

\newcommand{\beq}{\begin{equation}}
\newcommand{\eeq}{\end{equation}}

\newtheorem{defn}{Definition}[section]

\renewcommand{\deg}{\operatorname{deg}}

\newcommand{\diam}{\operatorname{diam}}
\newcommand{\dist}{\operatorname{dist}}

\newcommand{\eps}{\epsilon}

\numberwithin{equation}{section}

\newcommand{\cA}{{\mathcal A}}

\newcommand{\cM}{{\mathcal M}}

\renewcommand{\cR}{{\cal R}}

\renewcommand{\cD}{{\cal D}}

\newcommand{\CC}{{\mathbb C}}
\newcommand{\RR}{{\mathbb R}}

\newcommand{\ZZ}{{\mathbb Z}}
\newcommand{\NN}{{\mathbb N}}

\newcommand{\QQ}{{\mathbb Q}}

\newcommand{\ignore}[1]{{}}

    \title[Computational Complexity in the real quadratic family]{Computational Intractability of attractors in the real quadratic family }
\author{Cristobal Rojas and Michael Yampolsky}
\begin{document}
\maketitle

\begin{abstract}  We show that there exist real quadratic maps of the interval whose attractors are computationally intractable. This is the first known class of such natural examples.  
\end{abstract}


\section{Introduction}
A simple dynamical system, which is easy to implement numerically, can nevertheless exhibit chaotic dynamics.
This renders impractical attempting to compute the behaviour of a trajectory of the system for an extended 
period of time: small computational errors are magnified very rapidly. 
Thus, the modern paradigm of the numerical study of chaos is the following: since the simulation of an individual orbit for an extended
period of time does not make a practical sense, one should study the limit set of a typical orbit (both as a spatial object and as
a statistical distribution). Such limit sets are known as {\it attractors}, we refer the reader to \cite{Mil} for a detailed discussion of the
relevant definitions. 

From the theoretical computability point of view, the principal problem thus becomes:

\medskip
\noindent
{\it Suppose that a dynamical system with a single attractor can be numerically simulated. Can its attractor be effectively computed? }
\medskip
\noindent

The first author with M. Hoyrup and S. Galatolo constructed in \cite{HRG} a computable map of the unit circle for which the orbit of every point accumulates in a set that is not effectively computable. However, the dynamics restricted to this set is not transitive, and  the class of maps one obtains is rather artificial.

The second author and M.~Braverman had obtained a natural class of counter-examples in the setting of one-dimensional complex dynamics.
Recall, that for a rational map $R(z)$ of $\hat\C$ with $\deg R\geq 2$, the Julia set $J(R)$ is the {\it repeller} (that is,  the attractor for the multi-valued dynamics of $R^{-1}$).  In a series of works \cite{BY1,BY2,BY3} they showed that there exist quadratic polynomials $P_c(z)=z^2+c$ with {\it computable} values of $c$ whose Julia sets $J(P_c)$ cannot be effectively computed.

In this work we present a different class of examples which are even more striking. Indeed, they occur in the same quadratic family $P_c$ but this time with real values of $c$ and viewed as maps of the interval, as opposed to maps of the complex plane. The study of such dynamical systems, known as {\it unimodal maps} has been the cornerstone of one-dimensional dynamics, the subject that blossomed in the 1970's, and has been at the center of attention since. 

As the reader will see below, these maps have attractors (in the classical sense) with a well-understood topological structure. In contrast with Julia sets, given an access to the value of the parameter $c$, such an attractor is {\it always} computable. However, this is only true in theoretical terms. Our main result is:

\medskip
\noindent
    {\sl Given an arbitrary lower bound on time complexity $f:\NN\to\NN$, we can  produce a parameter $c$ such that any algorithm which computes the attractor of the unimodal map $P_c$ has a  running time worse  than $f$.}

    \medskip
    \noindent
Of course, for a sufficiently ``bad'' lower bound, this renders the computation impossible in practice.

Similarly to the case of non-computable quadratic Julia sets, our construction is quite delicate, and involves modern tools of Complex Dynamics, such as parabolic implosion and renormalization.


\section{Preliminaries}

 \subsection*{Computational Complexity of sets}
 
We give a very brief summary of relevant notions of Computability Theory and Computable Analysis. For a more in-depth
introduction, the reader is referred to e.g. \cite{BY3}.
As is standard in Computer Science, we formalize the notion of
an algorithm as a {\it Turing Machine} \cite{Tur}.  
Let us begin by giving the modern definition of the notion of computable real
number,  which goes back to the seminal paper of Turing \cite{Tur}. By identifying $\Q$ with $\N$ through some effective enumeration, we can assume algorithms can operate on $\Q$. Then a real number $x\in\RR$ is called \defin{computable} if there is an algorithm  $M$ which, upon input $n$, halts and outputs a rational number $q_n$ such that  $|q_n-x|<2^{-n}$.
Algebraic numbers or  the familiar constants such as $\pi$, $e$, or the Feigenbaum constant  are computable real numbers. However, the set of all computable real numbers $\RR_C$ is necessarily countable, as there are only countably many Turing Machines.

Computability of compact subsets of $\RR^k$ is defined by following the same principle.  Let us say that a point in $\RR^k$ is a {\it dyadic rational with denominator} $2^{-n}$ if it is of the form $\bar v\cdot 2^{-n}$, where $\bar v\in\ZZ^k$ and $n\in\NN$. 
Recall that {\it Hausdorff distance} between two
compact sets $K_1$, $K_2$ is
$$\dist_H(K_1,K_2)=\inf_\eps\{K_1\subset U_\eps(K_2)\text{ and }K_2\subset U_\eps(K_1)\},$$
where $U_\eps(K)=\bigcup_{z\in K}B(z,\eps)$ stands for an $\eps$-neighbourhood of a set.

\begin{defn}\label{1}We say that a compact set $K\Subset\RR^k$ is {\it computable} if there exists an algorithm $M$ with a single input $n\in\NN$, which outputs a
finite set $C_n \subset \Q$ of dyadic rational points in $\RR^k$ such that
$$\dist_H(C_n,K)<2^{-n}.$$
\end{defn}


An equivalent way of defining computability of sets is the following. For $\bar x=(x_1,\ldots,x_k)\in\RR^k$ let
the norm $||\bar x||_1$ be given by
$$||\bar x||_1=\max|x_i|.$$
\begin{defn}
  \label{def-local}
A compact set $K\Subset\RR^k$ is computable if there exists an algorithm $M$ with a single input $n\in\NN$ and  a dyadic rational point $x$ with denominator $2^{-n}$, such that the following holds.  $M$ outputs $0$ if $x$ is at least $2\cdot 2^{-n}$-far from $K$ in $||\cdot||_1$ norm, outputs $1$ if $x$ is at most
$2^{-n}$-far from $K$, and outputs either $0$ or $1$ in the ``borderline'' case.
\end{defn}
In the familiar context of $k=2$, such an algorithm can be used to ``zoom into'' the set $K$ on a computer screen with $W\times H$ square pixels to draw an accurate picture of the portion of $K$ inside a rectangle of width $W\cdot 2^{-n}$ and height $H\cdot 2^{-n}$. $M$ decides which pixels in this picture have to be black (if their centers are $2^{-n}$-close to $K$) or white (if their centers are $2\cdot 2^{-n}$-far from $K$), allowing for some ambiguity in the intermediate case.

Let $C=\dist_H(K,0)$. 
For an algorithm $M$ as in Definition~\ref{def-local} let us denote by $T_{M}(n)$ the supremum of running times of $M$ over all dyadic points with denominator $2^{-n}$ which are inside the ball of radius $2C$ centered at the origin: this is the computational cost of using $M$ for deciding the hardest pixel at the given resolution.

\begin{defn}
We say that a function $T:\NN\to\NN$ is a {\it lower bound} on time complexity of $K$ if for any $M$ as in Definition~\ref{def-local} there exists an infinite sequence $\{n_i\}$
such that
$$T_M(n_i)\geq T(n_i). $$
Similarly, we say that $T(n)$ is an {\it upper bound} on time complexity of $K$ if there exists an algorithm $M$ as in Definition~\ref{def-local} such that for all $n\in\NN$
$$T_M(n)\leq T(n).$$
\end{defn}

In this paper, we will be interested in the time complexity of attractors of  quadratic maps of the form $x^2 + c$, with $c\in \RR$.
As is standard in computing practice, we will assume that the algorithm can read the value of $c$ externally to produce a zoomed in picture of the attractor. More formally, let us denote $\cD_n\subset \RR$ the set of dyadic rational numbers with denominator $2^{-n}$. 
We say that a function $\phi:\NN\to\QQ$ is an {\it oracle} for $c\in \RR$ if for every $m\in \NN$
$$\phi(m)\in \cD_m\text{ and }d(\phi(m),c)<2^{-(m-1)}.$$
We amend our definitions of computability and complexity of a compact set $K$ by allowing {\it oracle Turing Machines} $M^\phi$ where
$\phi$ is any function as above. 
On each step of the algorithm, $M^\phi$ may read the value of $\phi(m)$ for an arbitrary $m\in\NN$.

This approach allows us to separate the questions of computability and computational complexity of a parameter $c$ from that of the attractor. It is crucial to note that reading the values of $\phi$ comes with a computational cost:

\medskip
\noindent
{\it querying $\phi$ with precision $m$ counts as $m$ time units. In other words, it takes
  $m$ ticks of the clock to read the first $m$ dyadic digits of $c$.
}

\medskip
\noindent
This is again in a full agreement with computing practice: to produce a verifiable picture of a set, we have to use the ``long arithmetic'' for constants, which are represented by sequences of dyadic bits. The computational cost grows with the precision of the computation, and manipulating a single bit takes one unit of machine time.


\subsection*{Attractors of quadratic maps of the interval and the statement of the main result}
Consider a real quadratic polynomial $P_c(x)=x^2+c$, with $c\in[-2,-1]$. We denote $$I_c\equiv [c,P_c(c)].$$
It is easy to see that $P_c(I_c)=I_c$; we will refer to the invariant interval $I_c$ as the {\it dynamical interval} of $P_c$.
We denote $\Omega(P_c)$ the postcritical set
$$\Omega(P_c)\equiv\overline{\cup_{n\geq 0}P_c^n(0) }.$$

Let us say that $P_c$ is {\it infinitely renormalizable} if there exists an infinite  nested sequence of cycles of periodic sub-intervals of $I_c$:
$$\cC_0\supset \cC_1\supset \cC_2\supset\cdots\supset \omega(0)$$ with increasing periods. We say that the {\it Feigenbaum-like Cantor set} of an infinitely
renormalizable polynomial is the intersection
$\cap_{k\in\NN}\cC_k.$
It is known (see \cite{G}) that:
\begin{theorem}
  Let $c\in[-2,-1]$ and let $P_c$ be infinitely renormalizable. Denote $\A$ the Feigenbaum-like Cantor set of $P_c$. Then $\A=\omega(0)$. Furthermore, for Lebesgue almost every $x\in I_c$,
the limit set $\omega(x)=\A$; the same is true for a dense-$G_\delta$ set of $x\in I_c$.
\end{theorem}
Thus, the Feigenbaum-like Cantor set is an attractor both in the measure-theoretic and in the topological sense (cf. \cite{Mil}). We will refer to it as the {\it Feigenbaum-like attractor}. 

It is worthwhile to note that there are only three possibilities for the structure of an attractor of a real quadratic polynomial (see \cite{Lyu3} where the classification was completed, and references therein):
\begin{theorem}
Let $P_c$ and $I_c$ be as above.
Then there is a unique set $\A$ (a measure-theoretic attractor in the sense of Milnor)
such that $\A = \omega(x)$ for Lebesgue almost all $x\in [0, 1]$, and only one of the following three possibilities can
occur:
\begin{enumerate}
\item $\A$ is a limit cycle;
\item $\A$ is a cycle of intervals;
\item $\A$ is a Feigenbaum-like attractor.
\end{enumerate}
In all of the above cases, $\A$ is also the topological attractor of $P_c$.
\end{theorem}
Our main result is the following. 

\bigskip
\noindent\textbf{Main Theorem.} \emph{
  The attractor of a quadratic map $P_c$ is always computable given a parameter $c$. However, given any function $f:\NN\to\NN$, there exists a value of $c$ such that the map $P_c$ has a Feigenbaum-like attractor $\A$, whose computational  complexity is bounded below by $f(n)$.
	}
\bigskip


\section{Combinatorics of renormalization}

\subsection*{Renormalization windows}
Given two points $a\neq b$, we will denote $[a,b]$ the closed interval connecting them without regard to their linear order.
For our purposes, {\it a unimodal map} of the interval $[a,b]\ni 0$ is an analytic map with a single extremum at $0$, such that
$f([a,b])=[a,b]$ and $a=f(0)$, $b=f^2(0)$. We call $I_f\equiv [a,b]$ the {\it dynamical interval } of $f$.  

We say that a unimodal map $f$ is {\it renormalizable} if there exists $n>1$ and a sub-interval $J\ni 0$ such that $$f^n:J\to J$$
is a unimodal map.
We call the lowest such $n$ the {\it period of renormalization}, and
write $$n\equiv p(f);$$
we call the corresponding $J$ {\it a renormalization interval}.
 The {\it renormalization } $\cR(f)$ is the unimodal map
$$\cR(f)=\Lambda^{-1}\circ f^n|_J\circ \Lambda,\text{ where }\Lambda(x)\equiv ((f|_J)^{n}(0)\cdot x). $$
The map
$f^n|_J$ can be renormalizable in its turn, and so on, giving a rize of a sequence of periods $1<n_1, n_2,\ldots$ and a nested sequence of renormalization intervals $J_1=J\supset J_2\supset\cdots$.
If this sequence is infinite, we call $f$ {\it infinitely renormalizable}.
Each $J_i$ is periodic with the period $n_1n_2\cdots n_i\equiv p_i$.
We denote $\cC_i$ and $\hat \cC_i$ the collections of intervals
$$\cC_i=\cup _{k=0}^{p_i-1}f(J_i);\text{ and }\hat\cC_i=\cup_{k=0}^{n_i-1}f^{p_{i-1}}(J_i).$$
Thus, $\hat\cC_i$ consists of the intervals of the cycle $\cC_i$ contained in the renormalization interval of the previous level $J_{i-1}$. The iterate
$f^{p_{i-1}}$ induces a permutation of the intervals $\hat\cC_i$; we will call the corresponding element of the symmetric group $S_{n_i}$ {\it the combinatorial type} of the $i$-th renormalization of $f$, and denote it $\tau_i(f)$.

Denote $W_\tau$ the set of real renormalizable quadratic polynomials with combinatorial type of the first renormalization equal to $\tau$. This set is a closed interval known as a {\it renormalization window}; it is equal to the intersection of a small copy of the Mandelbrot set $\M_\tau$ with the real line.

Let us denote $\chi$ the Douady-Hubbard {\it straightening map} \cite{DH}. For each quadratic-like map $f$ with a connected Julia set, it corresponds a unique parameter $c$ in the Mandelbrot set $\cM$ such that $f$ is hybrid equivalent to $P_c$. For $c\in W_\tau$, the mapping
$$c\mapsto \chi(\cR(P_c))$$
is a homeomorphism between $W_\tau$ and $[-2,0.25]=\cM\cap\RR$; it naturally extends to a homeomorphism $\cM_\tau\to\cM$.

Renormalization windows are dense in $[-2,0.25]$. Let $n$ be the renormalization period of $W_\tau$ (that is, $\tau\in S_n$). 
The left endpoint of  $W_\tau$ is the parameter $l$ such that the critical value $P^n_l(0)$ is a pre-fixed point of $P^n_l$:
\begin{equation}
  \label{left}
  P^{2n}_l(0)=P^{3n}(0).
  \end{equation}
The right end-point $b$ is the cusp of $\M_\tau$: the map $P^n_b$ has a parabolic fixed point with multiplier $1$. It is the latter case that will be at the center of our attention. We will see below that one can find two small renormalization windows $W_1$ and $W_2$ which are both arbitrarily close to $b$ on the left-hand side such that the postcritical sets of maps in these windows are drastically different. This is the well-studied phenomenon of {\it parabolic implosion}; we will review some of its applications below.

By way of example, consider the cyclical permutation $\tau=(2,3,1)$. The orbit of the renormalization interval for a map in $W_\tau$ is illustrated in the top portion of Figure~\ref{fig-tau}. The interval $W_\tau=[l,r]$ is the intersection of a small copy of the Mandelbrot set $\cM_\tau$ with the real line; this is illustrated in the bottom portion of the same figure. This is the unique small copy with period $3$ intersecting the real line; we will sometimes refer to it as $\cM^3$. The right end-point $r=-1.75$ corresponds to the polynomial with a periodic point of period $3$ with the multiplier equal to $1$.

\begin{figure}

\centerline{\includegraphics[width=\textwidth]{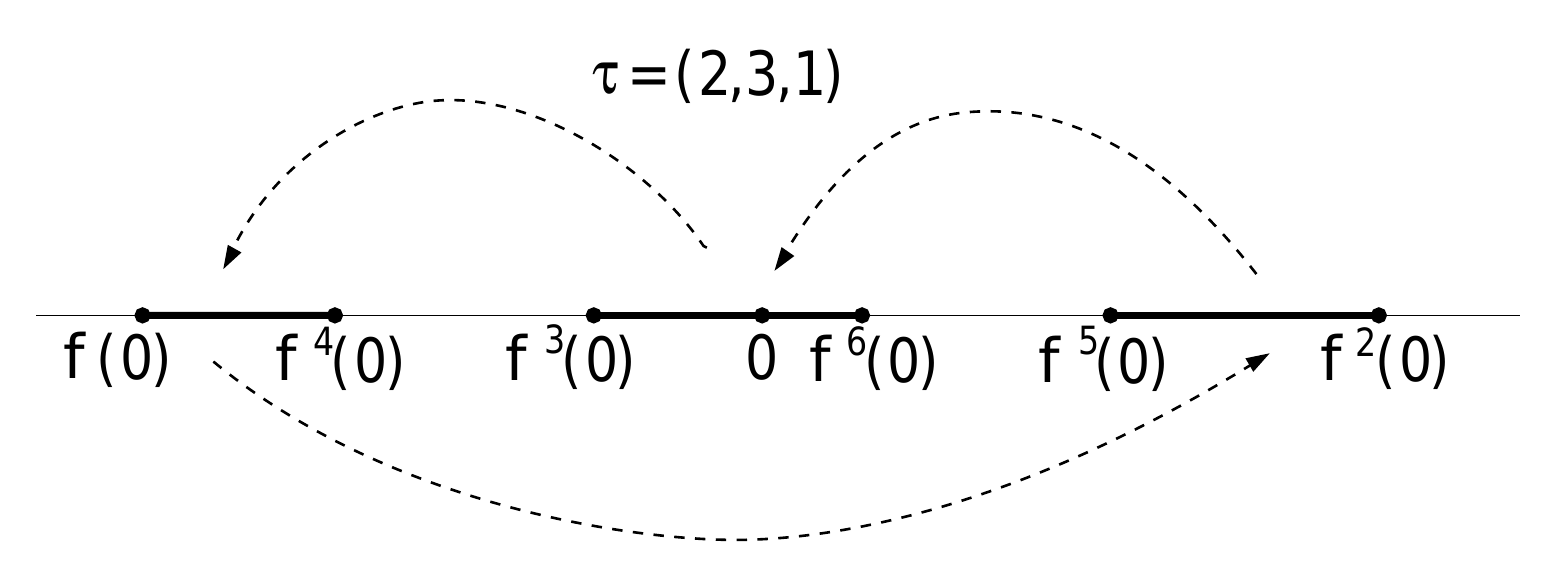}} 
\centerline{\includegraphics[width=\textwidth]{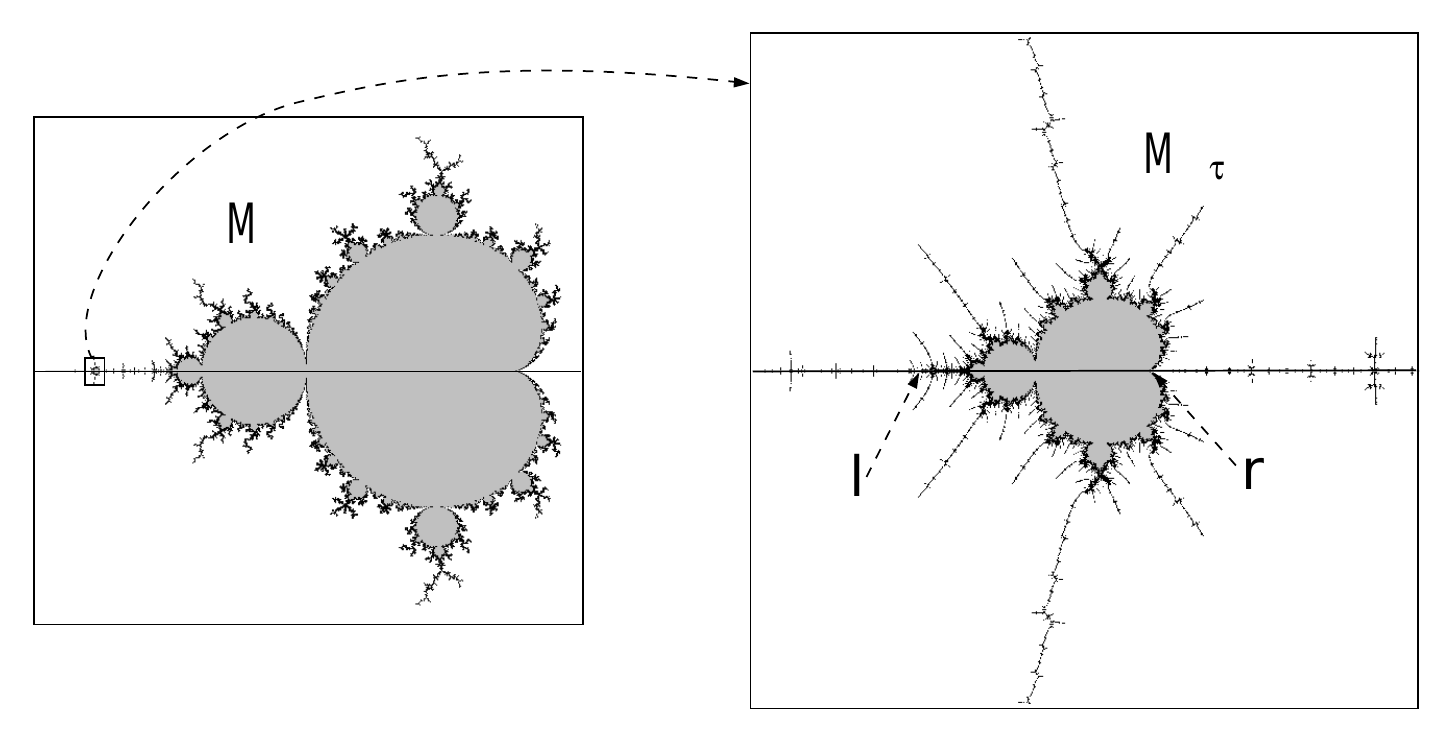}} 
\caption{Above: the orbit of the renormalization interval for the combinatorial type $\tau=(2,3,1)$. Below: the small copy $\cM_\tau$ inside the Mandelbrot set.
\label{fig-tau}}
 \end{figure}

\subsection*{Definition of the essential period}
A detailed discussion of the combinatorics of
renormalization goes beyond the scope of this paper. We will recall some of the relevant concepts briefly.

Let $f:I_f=[a,b]\to[a,b]$ be a renormalizable unimodal map.
Denote $\alpha_f\in I_f$ the fixed point of $f$. 
The {\it principal nest} of $f$ is the sequence 
of intervals 
$$[-\alpha_f,\alpha_f]\equiv I^0\supset I^1\supset I^2\supset\cdots$$
where $I^m\ni 0$ is the central
component of the first return map of $I^{m-1}$,
$$g_m:\cup I_i^m\to I^{m-1}.$$
A level $m>0$ is {\it non-central}, if $g_m(0)\in I^{m-1}\setminus I^m$.
If $m$ is non-central, then $g_{m+1}|_{I^{m+1}}$ is not merely a restriction of 
the central branch of $g_m$, but a different iterate of $f$. Set $m(0)=0$, and let
$$m(0)<m(1)<m(2)<\cdots<m(\kappa)$$
be the sequence of non-central levels. The map 
$$g_{m(\kappa)+1}|_{I^{m(\kappa)+1}}\equiv f^{n_1}.$$
For $0\leq k<\kappa$ the nested intervals
$$I^{m(k)+1}\supset I^{m(k)+2}\supset\cdots\supset I^{m(k+1)}$$
form a {\it central cascade}, whose {\it length} is $m(k+1)-m(k)$.
Lyubich called a cascade {\it saddle-node} if $0\notin g_{m(k)+1}(I^{m(k)+1})$.
The reason for this terminology is that
if the length of a saddle-node cascade is large, then $g_{m(k)+1}|_{I^{m(k)+1}}$
is combinatorially close to the saddle-node quadratic map $x\mapsto x^2+1/4$.

Let $x\in P(f)\cap (I^{m(k)}\setminus I^{m(k)+1})$ and set 
$d(x)=\min \{j-m(k),m(k+1)-j\}$, where $g_{m(k)+1}(x)\in I^j\setminus I^{j+1}$.
This number shows how deep the image of $x$ lands inside the cascade.
Let us now define $d_k$ as the maximum of $d_k(x)$ over all points $x\in P(f)\cap (I^{m(k)}\setminus I^{m(k)+1})$. For a saddle-node cascade the levels $l$ such that
$m(k)+d_k<l<m(k+1)-d_k$ are {\it neglectable}. 
Now we define the essential period of $f$ as follows. Set $J=I^{m(\kappa)+1}$, and
let $p$ be its period, that is the smallest positive integer  for which
$f^p(J)\ni 0$. Consider the orbit $J_0\equiv J$, $J_i=f^i(J_0)$, $i\leq p-1$. For each
$J_k$ consider the deepest cascade which contains this interval, and call $J_k$
neglectable if the cascade is saddle-node and $J_k$ is contained in a neglectable level
of the cascade. Now count the non-neglectable intervals in the orbit $\{J_i\}_{i=0}^{p-1}$.
Their number is the {\it essential period}, $p_e(f)$. 
Recall that an infinitely renormalizable map $f$ has a bounded combinatorial type if there is
a finite upper bound on the periods of its renormalizations. Similarly, 
$f$ is said to have an {\it essentially bounded combinatorial type} if 
$\sup_k p_e(\cR^k)<\infty$. 

We say that two renormalization types $\tau$ and $\tau'$ are {\it essentially equivalent} if removing the neglectable intervals from both renormalization cycles, we obtain the same permutation.

\medskip

\noindent
\subsection*{ An example of a map with essentially bounded combinatorics.}
The definiton given above is rather delicate. It is useful therefore to provide the
reader with a simple yet archetypical example of an 
infinitely renormalizable map of unbounded but essentially bounded combinatorial type 
(cf. \cite{Hinkle,Yam-bounds}). This map is constructed in such a way
that its every renormalization is a small perturbation of a unimodal map with a period 3 
parabolic orbit (see Figure~\ref{period3}). Closeness to a parabolic will ensure that the renormalization periods
are high, but the essential periods will all be bounded.

\begin{figure}
\begin{center}
\includegraphics[width=\textwidth]{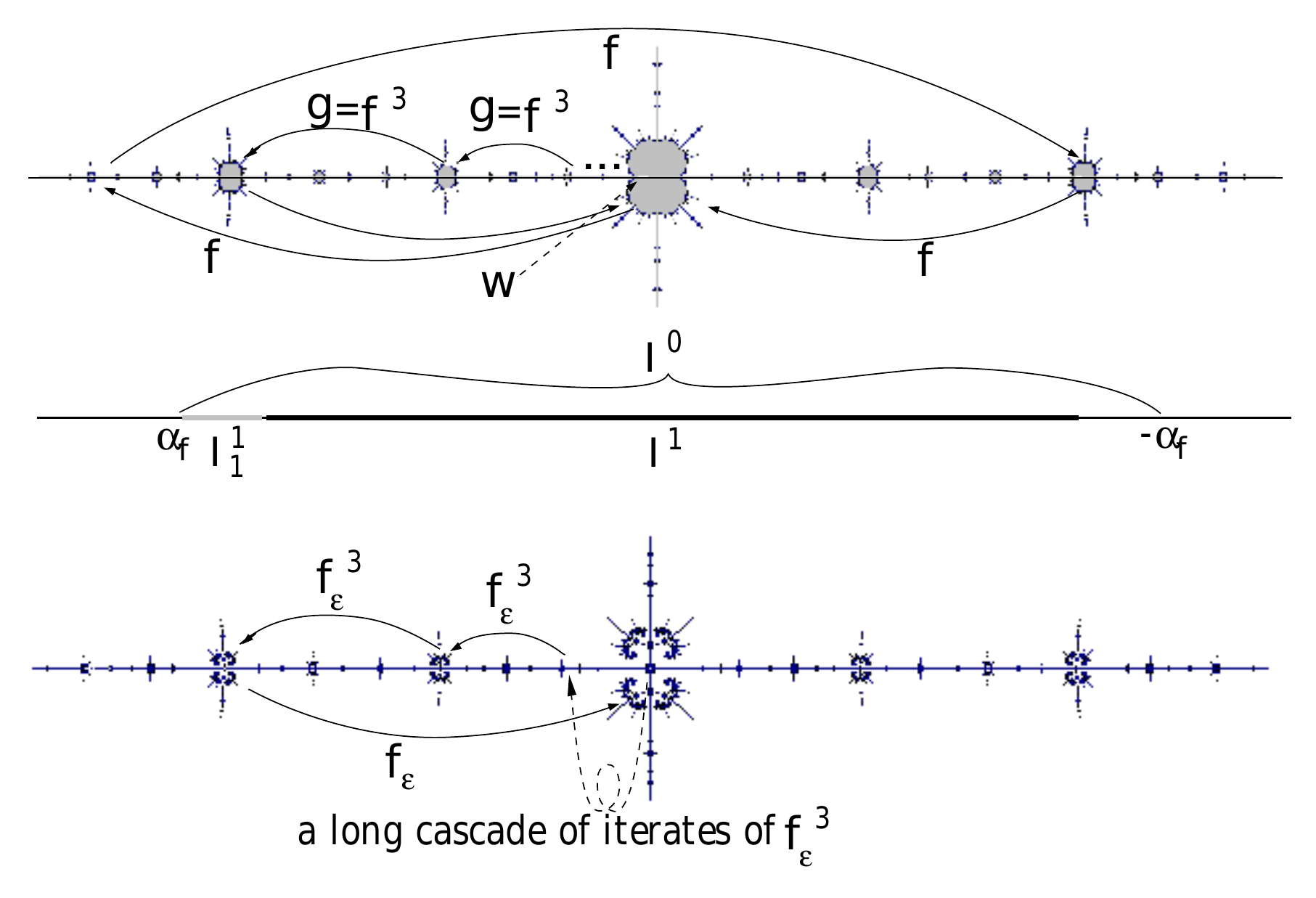} 
\caption{\label{period3}Above: the dynamics of the map $z\mapsto z^2-1.75$. Center: the domain of $g=f^3|_{I^1}$.
  Below: a small perturbation $f_\eps$, as in our example, with the orbit of the renormalization interval indicated. Note the long saddle-node cascade of iterates of $f_\eps^3$ which arises in the vicinity of the parabolic point $w$ of the unperturbed map $f$.
}
\end{center}
 \end{figure}

Before constructing the example, let us consider the dynamics of the quadratic map
$f:z\mapsto z^2-1.75$. This polynomial has a parabolic orbit of period $3$ on the real line,
let us denote $w$ the element of this orbit which is nearest to $0$. Recall that 
$I^0=[-\alpha_f,\alpha_f]$, and $I^1$ is the central component of the domain of the 
first return map $g:I^0\to I^0$. For this map we have $g|_{I^1}\equiv f^3$, 
$w\in I^0$, and $f^{3n}(0)\to w$. The map $g$ has two non-central components;
denoting $I^1_1$ the one whose boundary contains $\alpha_f$, we have 
$g=f^2:I^1_1\to I^0$. For a small $\epsilon>0$ let us set $f_\eps(z)=z^2-1.75+\eps$.
The orbit of $0$ under $f_\eps$ eventually escapes $I^0$. Let us define $\eps_n$
as the parameter value for which $$P^{3i}_{\eps_n}(0)\in I^1, \;i\leq n-1,$$
$$P_{\eps_n}^{3n}(0)\in I^1_1\text{, and }P_{\eps_n}^{3n+1}(0)=0.$$ These maps correspond to the
centers of a sequence of small copies $\cM_n^{3}$ of the Mandelbrot set converging
to he cusp $c=-1.75$ of the real period $3$ copy $\cM^{3}$. For each $P_{\eps_n}$ the
essential period $p_e(P_{\eps_n})=4$, obviously $p(P_{\eps_n})\to\infty$.
Now consider an infinitely renormalizable unimodal map $h$ such that the
combinatorial type $\tau(\cR^kh)=\tau(P_{\eps_{n_k}})$, with $n_k\to\infty$.
This is the desired example. We can, of course, select $h$ in the real quadratic family,
picking an infinitely renormalizable parameter value $c\in\cM$ such that
$\chi(\cR^k(f_c))\in\cM^{3}_{n_k}$. This amounts to blowing up a small copy $\cM^{3}_{n_1}$,
finding its period $3$ cusp, and the corresponding sequence of small copies converging
to this cusp, blowing up one of them, {\it ad infinitum} (see Figure~\ref{blow-up}).

\begin{figure}[phtb]
\centerline{\includegraphics[width=12cm]{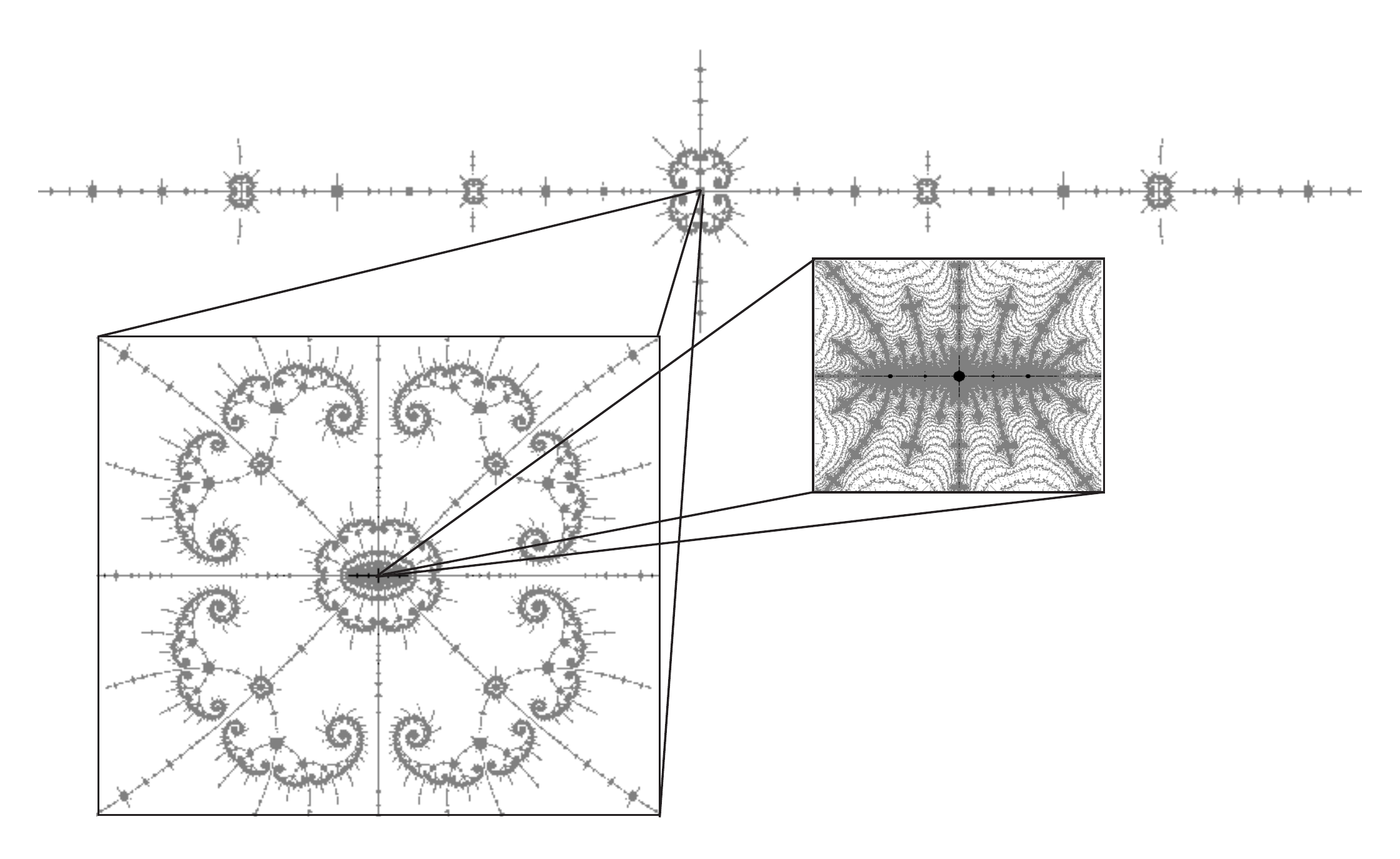}}
\caption{\label{blow-up}
An airplane inside of an airplane: consecutive blow-ups of a Julia set of 
a map with essentially bounded combinatorics, and the corresponding blow-ups of the 
Mandelbrot set}
\centerline{\includegraphics[width=12cm]{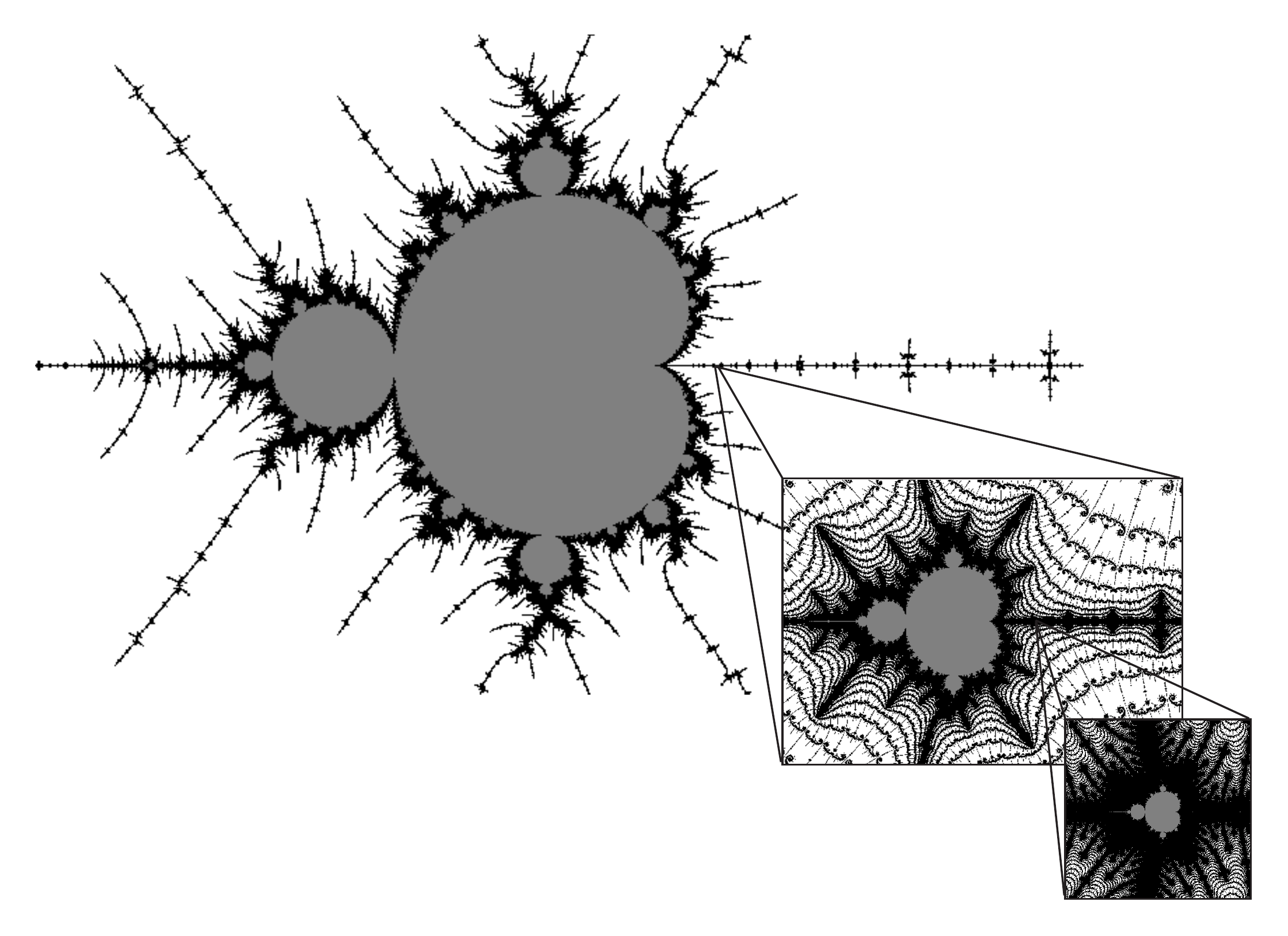}}
\end{figure}

\subsection*{Applications of parabolic implosion to limits of maps with essentially bounded combinatorics}
Theory of parabolic implosion is the principal mechanism used in our proof of the main result. It is quite involved and we will not attempt to give a self-contained
review here. For a beautiful introduction, see the paper of Douady \cite{Do}. The applications to dynamics of quadratic polynomials are described in the paper of the second
author \cite{Yam-bounds} and the work of Hinkle \cite{Hinkle}. Before giving a brief summary of the relevant results below, let us very informally describe their main thrust.
Consider a sequence of quadratic polynomials $P_{c_n}$ with $c_n\to c_*$. The limiting map $P_{c_*}$ can be described as the {\it algebraic limit} of the sequence $P_{c_n}$. The {\it geometric limit} of the same sequence consists of all of the analytic maps $\{ g\}$  which can be obtained as limits of uniformly converging subsequences of arbitrary iterates of our polynomials:
$$P^{m(n_k)}_{c_{n_k}}|_\Omega\rightrightarrows g,\text{ where }\Omega\text{ is a subdomain of }\CC.$$
Clearly, the geometric limit contains all of the iterates of $P_{c_*}$, but may {\it a priori} be larger. As an example, consider the parabolic quadratic polynomial
$P_{1/4}$. It has a fixed point $p=\frac{1}{2}$ with multiplier $1$, and its critical orbit
$$P_{1/4}^n(0)\nearrow \frac{1}{2}.$$
In particular, no iterate of $0$ under $P_{1/4}$ lies to the right of $\frac{1}{2}$. On the other hand, {\it for every} $\eps>0$
\begin{equation}
  \label{div}
  P_{1/4+\eps}^n(0)\nearrow \infty.
\end{equation}
An easy way to see this is to apply the coordinate change $w=z-\frac{1}{4}$, which transforms $P_{1/4+\eps}$ into
$$f_\eps(w)\equiv w+w^2+\eps;$$
clearly, $f^n_\eps(w)\geq w+ n\eps$ for all $w\in\RR$. 
Hence, for every  sequence $P_{1/4+\eps_n}$ with $1>\eps_n\to 0$ the geometric limit will contain maps $g$ which map $0$ to a point between $1$ and $2$. It will thus be larger than the algebraic limit. More importantly for our needs, this will be reflected in the fact that any Hausdorff limit point of the postcritical sets of $P_{1/4+\eps_n}$ will be larger than the postcritical set of $P_{1/4}$: in particular, it will contain points in the interval $[1,2]$. The theory of parabolic implosion provides a
description of such Hausdorff limit points, and relates them to particular sequences of perturbations of parabolic parameter values.

We now proceed to quote several facts we are going to need. All of them are consequences of the main rigidity result of \cite{Hinkle}, which is, in turn, a version of the Rigidity Theorem for parabolic towers of \cite{Ep1} (a formal statement of the Tower Rigidity theorem is beyond the scope of the present paper).

\begin{theorem}
\label{thm:sequence1}
Let $n\in\NN$. Suppose $c_k$ is a sequence of
parameter values in $\cM$ such that the following properties hold.
\begin{itemize}
 \item  Each $P_{c_k}$ is $n$-times renormalizable, and
   for each $j\leq n$ we have $\tau_j(P_{c_k})\equiv \tau_j$ are identical.
 \item Furthermore, the combinatorial types
   $\tau_{n+1}(P_{c_k})$ are essentially equivalent, and periods $p(\cR^n(P_{c_k}))\to\infty$.
 \item For each $k$ the renormalization
   $\cR^{n}(P_{c_k})$ has a single saddle-node cascade, whose period is greater than the essential period $p_e(\cR^n(P_{c_k}))$.
 \item Finally, assume that
$\chi(\cR^{n+1}(P_{c_k}))$ does not depend on $k$ and is a parabolic parameter ${c_*}$.
\end{itemize}
Then we have:
\begin{enumerate}
\item the parameters $c_k$ have a limit, which is the cusp of a small copy of the Mandelbrot set;
\item the postcritical sets of $P_{c_k}$ have a limit, which only depends on $c_*$ and is different for different values of $c_*$.
\end{enumerate}

\end{theorem}

\noindent
We also note
\begin{theorem}
  \label{thm:sequence2}
  Suppose $\tau_n$ is a sequence of essentially equivalent combinatorial types with a single saddle-node cascade whose period is greater than $p_e(\tau_n)$ and whose periods  $p(\tau_n)\underset{n\to\infty}{\longrightarrow}\infty$. Then
  the renormalization windows $W_{\tau_n}$ converge to a parabolic parameter $\hat{c}$.
 
\end{theorem}

\noindent
Finally,
\begin{theorem}
  \label{thm:sequence3}
  For $i=1,2$, let $c_k^i$ be two different sequences of parameter values in $\cM$ satisfying Theorem~\ref{thm:sequence1} for the same $n$ and $c_*$.
  Furthermore, assume that for al $j\leq n$, the combinatorial types $\tau_j(c_k^i)$ are identical, and that
  $\tau_{n+1}(c_k^1)$ is not essentially equivalent to $\tau_{n+1}(c_k^2)$. Then the limits of the postcritical sets of $P_{c_k^i}$ are different for $i=1,2$.
\end{theorem}

\section{Proof of the Main Theorem}
\label{section:proof}

\subsection{Computability of the attractor $\cA$}
In this sub-section we will show that the attractor $\cA$ of the map $P_c:I_c\to I_c$, $c\in[-2,4]$ is always computable by a machine $M^\phi$ with an oracle for $c$. Recall that there are three possible types for $\cA$:
\begin{enumerate}
\item $\cA$ is a limit cycle $$w_0\overset{P_c}{\mapsto} w_1\overset{P_c}{\mapsto}\cdots\overset{P_c}{\mapsto}w_{n-1}\overset{P_c}{\mapsto}w_n=w_0.$$
Note that in this case, the cycle is either (a) attracting: $0<|DP_c^n(w_0)|<1$,  (b) super-attracting: $|DP_c^n(w_0)|=0$, or (c) parabolic: $DP_c^n(w_0)=\pm 1$. 
\item $\cA$ is a periodic cycle of intevals $$J_0\overset{P_c}{\mapsto} J_1\overset{P_c}{\mapsto}\cdots\overset{P_c}{\mapsto}J_{n-1}\overset{P_c}{\mapsto}J_n=J_0\text{ with }0\in J_0.$$
  In this case, $P_c$ is $k$-times renormalizable and $J_0$ is the renormalization interval of level $k$.
  \item $P_c$ is infinitely renormalizable and $\cA$ is its Feigenbaum-like Cantor set.
\end{enumerate}
We prove the following:
\begin{theorem}
  \label{thm:comput}
  The attractor $\cA$ is always computable by a Turing machine with an oracle for $c$. Moreover:
  \begin{itemize}
  \item all attractors of type (1a) are uniformly computable;
  \item all attractors of types (1b) and (1c) are computable by an oracle  machine $M^\phi$ which uses as non-uniform information the period $n$ of the cycle;
  \item all attractors of type (2) are computable by an oracle  machine $M^\phi$ which uses as non-uniform information the period $n$ of the cycle;
  \item all attractors of type (3) are uniformly computable.
  \end{itemize}
\end{theorem}  
\begin{proof}
{\sl Case (1a).} \\
  We run a brute force search to find a round disk $D$ with a dyadic radius centered at a dyadic point $x\in I_c$ and $n\in\NN$ such that
  $P^n_c(D)$ is univalent (which can be verified using the Argument Principle) and $P_c^n(D)\Subset D$. Such a disk will always exist, since any sufficiently small disk centered at a point of the attracting limit cycle has this property. By Schwarz Lemma, $D$ contains an attracting periodic point, and 
  $$\diam(P_c^{nk}(D))\overset{k\to\infty}{\longrightarrow}0$$
  geometrically fast. Since $P_c$ can have at most one non-repelling orbit, the images $P_c^{nk}$ converge to a point in the limit cycle $\cA$.
  We iterate $P_c^n$ on $D$ until $P_c^{nk}(D)$ is small enough, to  find a sufficiently good approximation of the cycle.

  \noindent
      {\sl Case (1b).} \\
      In this case, the critical point $0$ lies in the cycle, and hence the proof is trivial.

      \noindent
          {\sl Case (1c).}\\
We can use Weyl's algorithm \cite{Wey} to find all roots of the equation $P^n_c(w)=w$. An obvious modification of the argument from Case(1a) allows us to compute all {\it repelling} cycles of periods $\leq n$. 
Since all orbits of $P_c$ except for the limit cycle are repelling, the only orbit that is left is the limit cycle $\cA$.

\noindent
    {\sl Case (2).}\\
    In this case we have that $\cA = \cup_{i=0}^{n} J_{i}$  where $$J_0\overset{P_c}{\mapsto} J_1\overset{P_c}{\mapsto}\cdots\overset{P_c}{\mapsto}J_{n-1}\overset{P_c}{\mapsto}J_n=J_0.$$   $\cA$ is therefore clearly computable since $J_{0}=[P^{n}_{c}(0),P^{2n}_{c}(0)]$.  

    \noindent
    {\sl Case (3).}\\
A renormalization window $W_\tau$ with period $p$ is bounded by parameters $a$, $b$ for which $\chi(\cR(P_c))=-2$ and $\chi(\cR(P_c))=0.25$ respectively. These can be identified algebraically. Indeed, $P_{-2}(z)=z^2-2$ is a Chebyshev polynomial for which $0\mapsto -2\mapsto 2\mapsto 2$ and is the unique quadratic map with this combinatorics of the critical orbit; and $P_{0.25}$ is a parabolic map with a fixed point which has multiplier $1$ and is, again, the only such quadratic map.  Thus, for $P_a^p$ the critical point $0$ is pre-fixed, and $P_b^p$ has a parabolic fixed point. Since there are only finitely many such parameters for each $p$, we can compute all renormalization windows $W_\tau$ with a given period $p$ with an arbitrary precision.

Using an exhaustive search, we can thus identify the renormalization window $W_{\tau_1}\ni c$ with the smallest period. Proceeding inductively, we can identify the renormalization window $W_{\tau_2}\subset W_{\tau 1}$ which corresponds to the $2$-nd renormalization $\cR^2(P_c)$ and so on. At each renormalization level, we can then compute the corresponding period, say $p$, from which we can compute the corresponding cycle of period intervals as in Case (2). We proceed as above to find one-by-one $2^{-(n+2)}$-approximations of the nested cycles of periodic intervals $\cC_k$. We halt when we obtain a set for which every connected component has diameter $\leq 2^{-(n+2)}$: it is the desired $2^{-n}$-approximation of the Feigenbaum-like Cantor set $\cA$.

\end{proof}

\subsection{Constructing Feigenbaum-like sets with high complexity}
Without loss of generality, we can specialize to the case of a monotone function $f_n:\NN\nearrow \NN$.
There are countably many Turing Machines, and we begin by enumerating them in some arbitrary computable fashion: $M_1^\phi$, $M_2^\phi,\ldots$ so that every machine appears infinitely many times in the enumeration. For $i=1,2$ let $\hat\tau^i_n$ be an infinite sequence of combinatorial types such that:
\begin{itemize}
\item all $\hat\tau_n^i$ for the same value of $i=1,2$ are essentially equivalent;
\item $\hat\tau_n^1$ is not essentially equivalent to $\hat\tau_n^2$;
\item $\hat\tau_n^i$ has a single saddle-node cascade, whose period is greater than the essential period of $\hat\tau_n^i$;
\item periods $p(\hat\tau_n^i)\underset{n\to\infty}{\longrightarrow}\infty$;
\item for each $i=1,2$, the sequence of renormalization windows $W_{\hat\tau^i_n}$ converges to $-1.75$.

\end{itemize}
We will proceed constructing the value of $c$ inductively. At step $n$ of the induction, we will have a parabolic parameter $c_n$, and a natural number $l_n$ such that:
\begin{enumerate}
\item $|c_n-c_{n+1}|<2^{-3l_n}$;
\item $d(\Omega(P_{c_n}),\Omega(P_{c_{n+1}}))<2^{-3l_n}$;
\item given an oracle for $c_n$, the machine $M_n^\phi$ cannot compute a $2\cdot 2^{-l_n}$-approximation of $\Omega(P_{c_n})$ in time $f(l_n)$. More precisely, $M_n^\phi$ either does not halt in time $f(l_n)$, or outputs a set $K$ such that
  $$d(K,\Omega(P_{c_n}))>3\cdot 2^{-l_n};$$
\item $P_{c_n}$ is $n$ times renormalizable;
\item fix $j\in\NN$ and let $m,n\geq j$. Then $\tau_j(P_{c_n})=\tau_j(P_{c_m})$.
\end{enumerate}

\noindent
    {\bf Base of induction.} For $i=1,2$ and $n\in\NN$ let $c^i_n$ be renormalizable parameters such that $P_{i,n}\equiv P_{c^i_n}$ has the properties
    \begin{itemize}
    \item $\tau(P_{i,n})=\hat\tau^i_n$;
    \item $\chi(\cR(P_{i,n}))=-1.75$.
    \end{itemize}
    By Theorem~\ref{thm:sequence1}, for each $i=1,2$ there exists a Hausdorff limit $\omega_i$ of the postcritical set $\Omega(P_{c^i_n})$.
    By Theorem~\ref{thm:sequence3}, there exists $l_1\in\NN$ such that the distance
    \begin{equation}\label{eq1}
      \dist(\omega_1,\omega_2)>4\cdot 2^{-l_1}.
    \end{equation}
    We let the machine $M_1^\phi$ compute the attractor $\cA$ with precision $2^{-l_1}$, giving it $c=-1.75$ as the parameter.
    
The first possibility we consider is that the machine does not 
halt in the time $f(l_1)$. Then we set $c_1\equiv c^1_n$ such that  $|c^1_n+1.75|<2^{-f(l_1)}$. Note, that in the running time $f(l_1)$, the machine $M^\phi_1$ cannot tell the difference between these parameters, and therefore, it will not halt in the time $f(l_1)$.

The second possibility is that the machine does halt and outputs a set $A_1$. By (\ref{eq1}), there exist $i\in\{1,2\}$ and $n\in\NN$ such that
$|c^i_n+1.75|<2^{-f(l_1)}$ and
$$\dist(\Omega(c^i_n),A_1)>2\cdot 2^{-l_1}.$$
We set $c_1\equiv c^i_n$.
   
\medskip
\noindent
    {\bf Step of induction.}
    By Theorem~\ref{thm:sequence1}, there exist two sequences of parabolic parameter values $c^1_k$, $c^2_k$ which converge to $c_n$, and such that the combinatorial type of $\cR^n(P_{c^i_k})$ is equal to $\hat \tau^i_k$ and
    $$\chi(\cR^{n+1}(P_{c^i_k})=-1.75.$$
    By Theorems \ref{thm:sequence1} and \ref{thm:sequence3}, the
    corresponding sequences of postcritical sets $\Omega(P_{c^i_k})$ converge to different limits $\Omega^1\neq \Omega^2$. Let $l_{n+1}\geq l_n+1$
    be such that
    $$\dist (\Omega_1,\Omega_2)>2\cdot 2^{-l_{n+1}}.$$
    Now the argument proceeds as above. If the machine $M^\phi_{n+1}$ does not halt in time $f(l_{n+1})$, then, by continuity, there exists
    $k\in\NN$ such that the value $c_{n+1}=c_k^1$ satisfies the conditions (1)-(2) of the induction, and we make it our choice.

    Otherwise, if the machine $M^\phi_{n+1}$ does halt and outputs a set $K$, then there is $i\in 1,2$ and $K\in\NN$ such that for all $k\geq K$
    the property (3) is satisfied for $c_{n+1}=c^i_k$. We select $k$ large enough so that (1)-(2) are satisfied as well.

\newpage

\bibliographystyle{plain}

\end{document}